\documentclass[a4paper]{article}

\usepackage[utf8]{inputenc}
\usepackage[T1]{fontenc}

\usepackage{amstext}
\usepackage{amsfonts}
\usepackage{amsmath}
\usepackage{amssymb}
\usepackage{amsthm}
\usepackage{ragged2e}
\usepackage{tabularx}

\newtheorem{theorem}{Theorem}
\newtheorem{proposition}[theorem]{Proposition}
\newtheorem{lemma}[theorem]{Lemma}

\newtheorem{corollary}[theorem]{Corollary}

\newtheorem*{theoreme}{Theorem}
\newtheorem*{remark}{Remark}
\newtheorem*{notation}{Notation}
\newcommand{\R}{\mathbf{R}}

\newcommand{\BMS}{m_{\mathrm{BMS}}}
\newcommand{\BR}{\mathrm{BR}}

\newcommand{\Hr}{\mathbf{H}}

\title{The case of equality in the dichotomy of Mohammadi-Oh}
\author{Laurent Dufloux}

\begin{document}
\maketitle
\abstract{If $n \geq 3$ and $\Gamma$ is a convex-cocompact Zariski-dense discrete subgroup of $\mathbf{SO}^o(1,n+1)$ such that $\delta_\Gamma=n-m$ where $m$ is an integer, $1 \leq m \leq n-1$, we show that for any $m$-dimensional subgroup $U$ in the horospheric group $N$, the Burger-Roblin measure associated to $\Gamma$ on the quotient of the frame bundle is $U$-recurrent.}

\section{Introduction}
\subsection{Notations}
We fix once and for all an integer $n \geq 2$. Let $G=\mathbf{SO}^o(1,n+1)$, this is the group of direct isometries of the real $(n+1)$-dimensional hyperbolic space $\Hr^{n+1}$. Its acts conformally on the boundary $\partial \Hr^{n+1}$.

Recall the Busemann function
\[ b_\xi(x,y) = \lim_{t \to \infty} d(x,\xi_t)-d(y,\xi_t) \qquad \xi \in \partial \Hr^{n+1},\quad x,y \in \Hr^{n+1} \]
where $t \mapsto \xi_t$ is some geodesic with positive endpoint $\xi$.

Fix an Iwasawa decomposition $G=KAN$; recall that the maximal compact subgroup $K$ is isomorphic to $\mathbf{SO}(n+1)$, whereas the Cartan subgroup $A$ is isomorphic to $\R$ (since $G$ has rank $1$) and the maximal unipotent subgroup $N$ is isomorphic to $\R^n$.

Denote by $M$ the centralizer of $A$ in $K$; $M$ is isomorphic to $\mathbf{SO}(n)$. Recall that $M$ normalizes $N$ and there are isomorphisms $M \simeq \mathbf{SO}(n)$, $N \simeq \R^n$ such that the operation of $M$ on $N$ by conjugation identifies with the natural operation of $\mathbf{SO}(n)$ on $\R^n$.

We will always tacitly endow $N$ with the corresponding Euclidean metric.

Let $\Gamma$ be a discrete non-elementary subgroup of $G$. Throughout this paper we make the standing assumptions that

\begin{center} {\bf $\Gamma$ is Zariski-dense and has finite Bowen-Margulis-Sullivan measure.}
\end{center}
In fact except in the  last paragraph we will always assume that $\Gamma$ is convex-cocompact (this is stronger than finiteness of the Bowen-Margulis-Sullivan measure).

As usual, we denote by $\delta_\Gamma$ the growth exponent (also called Poincar\'{e} exponent) of $\Gamma$
\[ \delta_\Gamma = \limsup_{R \to \infty} \frac{\log \mathrm{Card} \{ \gamma \in \Gamma\ ;\ d(x,\gamma x) \leq R \} }{R} \]
which does not depend on the fixed point $x \in \Hr^{n+1}$. This is the Hausdorff dimension (with respect to the spherical metric on the boundary) of the limit set
\[ \Lambda_\Gamma = \overline{\Gamma \cdot x} \cap \partial \Hr^{n+1} \]
(which also does not depend on $x$). Bear in mind that $0 < \delta_\Gamma \leq n$; in this paper we will be interested in the case when $\delta_\Gamma$ is an integer strictly less than $n$.

The boundary $\partial \Hr^{n+1}$ is endowed with the Patterson-Sullivan density $(\mu_x)_{x \in \Hr^{n+1}}$. This is the (essentially unique since $\Gamma$ has finite Bowen-Margulis-Sullivan measure) family of finite Borel measures on $\partial \Hr^{n+1}$ satisfying
\begin{enumerate}
\item $\Gamma$-equivariance : $\mu_{\gamma x}$ is the push-forward of $\mu_x$ through the mapping induced by $\gamma$ on $\partial \Hr^{n+1}$;
\item $\delta_\Gamma$-conformality: for any $x,y \in \Hr^{n+1}$, $\mu_x$ and $\mu_y$ are equivalent and the Radon-Nikodym cocycle is given by
  \[ \frac{\mathrm{d} \mu_y}{\mathrm{d} \mu_x} (\xi) = e^{-\delta_\Gamma b_\xi(y,x)} \]
  almost everywhere.
\end{enumerate}
This is the Patterson-Sullivan density associated to $\Gamma$. If a base point $o \in \Hr^{n+1}$ is fixed, the boundary $\partial \Hr^{n+1}$ may be identified canonically with the $n$-sphere $S^{n}$ and thus endowed with the usual spherical metric. When $\Gamma$ is convex-cocompact, $\mu_o$ is proportional to  the $\delta_\Gamma$-dimensional Hausdorff measure on $\delta_\Gamma$ with respect to the spherical metric (see \cite{Sullivan} or \cite{Coornaert}).

We now recall the definition of the Bowen-Margulis-Sullivan (BMS) measure -- first on the unit tangent bundle, then on the frame bundle. Let $T^1 \Hr^{n+1}$ be the unit tangent bundle over $\Hr^{n+1}$. The Hopf isomorphism is the bijective mapping from $T^1 \Hr^{n+1}$ to $\partial^2 \Hr^{n+1} \times \R$ that maps the unit tangent vector $u$ with base point $x$ to the triple
\[ (\xi,\eta,s)= (u^-,u^+, b_{u^-}(x,o)) \]
where $u^-, u^+$ respectively are the negative and positive endpoints of the geodesic whose derivative at $t=0$ is $u$. The notation $\partial^2 \Hr^{n+1}$ stands for the set of all $(\xi,\eta) \in \partial \Hr^{n+1} \times \partial \Hr^{n+1}$ such that $\xi \neq \eta$.

In these coordinates, the BMS measure on $T^1 \Hr^{n+1}$ is given by
\[ \mathrm{d} \tilde{m}_{\mathrm{BMS}}(u) = e^{\delta_\Gamma (b_\xi(x,u)+b_\eta(x,u))} \mathrm{d} \mu_x (\xi) \mathrm{d} \mu_x (\eta) \mathrm{d}s \]
(it does not depend on the choice of $x \in \Hr^{n+1}$).

The BMS measure is a Radon measure  that is invariant under the geodesic flow as well as under the natural operation of $\Gamma$. The quotient of this measure with respect to $\Gamma$ is a Radon measure $\BMS$ on $\Gamma \backslash T^1\Hr^{n+1}$ that is still invariant with respect to the geodesic flow. This quotient measure may be finite or infinite; we will always assume that is is finite and in fact we will usually assume that it is compactly supported, which is equivalent to $\Gamma$ being convex-cocompact (\cite{NegCurv}, \cite{Sullivan}).

The Burger-Roblin (BR) measure is defined in a similar fashion:
\[ \mathrm{d} \tilde{m}_{\mathrm{BR}}(u) = e^{\delta_\Gamma b_\xi(x,u)+nb_\eta(x,u)} \mathrm{d}\mu_x(\xi) \mathrm{d} \nu_x(\eta) \mathrm{d} s\]
where $\nu_x$ is the unique  Borel probability measure on $\partial \Hr^{n+1}$ that is invariant under the stabilizer of $x$ in $G$; if $\partial \Hr^{n+1}$ is identified with $S^n$ accordingly, this is just the Lebesgue measure on $S^n$.

Likewise, the Burger-Roblin measure is $\Gamma$-invariant and thus defines a Radon measure on $\Gamma \backslash T^1 \Hr^{n+1}$. This Radon measure is always infinite, unless $\Gamma$ is a lattice.

Both these measures lift to the frame bundle over $\Gamma \backslash \Hr^{n+1}$ in the following way. The hyperbolic space $\Hr^{n+1}$ identifies with the quotient space $G/M$ so that $G$ identifies with the $(n+1)$-frame bundle over $\Hr^{n+1}$. The quotient space $\Gamma \backslash G$ accordingly identifies with the $(n+1)$-frame bundle over the orbifold $\Gamma \backslash \Hr^{n+1}$. There is a unique measure on $\Gamma \backslash G$ that is (right) invariant with respect to $M$ and  projects onto the BMS measure in $\Gamma \backslash G/M$, we denote it by $\BMS$ as well. Same thing for the BR measure. The lift of the geodesic flow to $\Gamma \backslash G$ is called the frame flow.

The point in doing this is we can now let $N$ act by translation (to the right) on $\Gamma \backslash G$. Let us agree that  $A = \{ a_t\ ;\ t \in \R \}$ where $(a_t)_t$ parametrizes the frame flow over $\Hr^{n+1}$, in such a way that $N$ parametrizes the \emph{unstable} horospheres. 

We then have, for every $h \in N$,
\begin{equation} \label{form.equivariance} a_{-t}h a_t=S_t(h)
\end{equation}
where $S_t$ is the homothety $N \to N$ with ratio $e^t$.

We summarize the important points in the following
\begin{lemma}
  Assume that $\Gamma$ has finite BMS measure and is Zariski-dense.
  \begin{enumerate}
  \item The BMS measure on $\Gamma \backslash G$ is mixing with respect to the ergodic flow.
  \item The BR measure on $\Gamma \backslash G$ is invariant and ergodic with respect to $N$.
    \item If $\Omega \subset \Gamma \backslash G$ has full BMS measure, then $\Omega N$ has full BR measure.
  \end{enumerate}
\end{lemma}
\begin{proof}
For 1 and 2 see \cite{Winter}. For 3 compare the definitions of BMS and BR measure, taking into account the fact that $N$ parametrizes the unstable horospheres in the frame bundle.
\end{proof}

\subsection{Background}
The basic motivation for this paper is the following
\begin{theoreme}[Mohammadi-Oh, \cite{OhMohammadi}, Theorem 1.1]
Assume that $\Gamma$ is convex-cocompact and Zariski-dense. Let $m$ be an integer, $1 \leq m \leq n-1$, and $U$ be an $m$-plane in $N$. If $\delta_\Gamma > n-m$, then $m_\BR$ is $U$-ergodic.
\end{theoreme}
This result was also obtained by Maucourant and Schapira \cite{MaucourantSchapira} under the weaker hypothesis that $\Gamma$ has finite BMS measure. The case when $\delta_\Gamma < n-m$ has also been settled by these authors:
\begin{theoreme}[Maucourant-Schapira, \cite{MaucourantSchapira}]
  Assume that $\Gamma$ is convex-cocompact and Zariski-dense. Let $m$ be an integr, $1 \leq m \leq n-1$, and $U$ be an $m$-plane in $N$. If $\delta_\Gamma < n-m$, then $m_\BR$ is totall $U$-dissipative. In particular, it is not ergodic.
\end{theoreme}

Mohammadi-Oh and Maucourant-Schapira  use Marstrand's projection Theorem to look at the geometry of the BMS measure along $U$ and $N$. For more on this, see \cite{preprintOM}.

In this paper, we use Besicovitch-Federer's projection theorem to study the case $\delta_\Gamma = n-m$. Our main result is the following

\begin{theoreme}
Assume that $\Gamma$ is convex-cocompact and Zariski-dense. Let $m$ be an integer, $1 \leq m \leq n-1$. If $\delta_\Gamma=n-m$, then the Burger-Roblin measure is recurrent with respect to any $m$-plane $U$ in $N$.
\end{theoreme}
Whether the BR measure is ergodic with respect to $U$ under these hypotheses remains an open question. We will see that the return rate of $U$-orbits is quite low (\emph{i.e.} subexponential) but this does not contradict ergodicity since BR is not finite.

Let us mention that the Theorem is not empty; indeed it is possible to construct some Zariski-dense convex-cocompact group $\Gamma \subset \mathbf{SO}^o(1,3)$ with $\delta_\Gamma=1$. Start with the Apollonian gasket associated to 4 mutually tangent circles on the boundary of $\Hr^3$; the limit set has dimension $\delta_\Gamma > 1$. Now shrink continuously the radii of the circles, thus lowering continuously $\delta_\Gamma$. The deformed group will remain Zariski-dense because the centers of the circles are not aligned. For details see \cite{McMullen}.

With this result for $\delta_\Gamma=n-m$, the situation is summarized in the following table. We assume that $\Gamma$ is Zariski-dense, has finite BMS measure, and we fix some $m$-plane $U$ in $N$ with $1 \leq m \leq n-1$. With respect to $U$, the BMS and BR measures are:
\begin{center}\begin{tabularx}{10.5cm}{|l|X|X|X|}
  \cline{2-4} \multicolumn{1}{r|}{} & $\delta_\Gamma < n-m$ & $\delta_\Gamma=n-m$ & $\delta_\Gamma>n-m$ \\
  \hline BMS & dissipative \cite{preprintOM} &  dissipative  \cite{preprintOM} & recurrent and ergodic \cite{MaucourantSchapira} \\
  \hline BR & totally dissipative \newline if $\Gamma$ convex-cocompact \cite{MaucourantSchapira} & recurrent if $\Gamma$ convex-cocompact & recurrent  and ergodic \cite{MaucourantSchapira}  \\ \hline
\end{tabularx}
\end{center}

Note that it follows immediately from the definitions that if the BMS measure is recurrent, so is the BR measure. The other implications are not so obvious.

We now sketch briefly our argument. In order to prove that the BR measure is $U$-recurrent (where $U$ is some $m$-plane), we need to show that the $U$-orbit of $m_{\mathrm{BR}}$-almost every $x \in \Gamma \backslash G$ will pass through some compact set $K$ infinitely often. If is enough to construct some sequence $h_k$ in $N$ that goes to infinity while staying uniformly close from $U$, such that $xh_k \in K$; indeed, if $u_k$ is the orthogonal projection of $h_k$ onto $U$, the sequence $u_k$ still goes to infinity and $x u_k$ will belong to some compact $K'$ that is just slightly bigger than $K$.

To show that such a sequence $(h_k)_k$ exists, our strategy is to prove that any $\rho$-neighbourhood of $U$ in $N$ has infinite measure with respect to the conditional measure of $\BMS$ along $N$; we then use the fact that the support of $\BMS$ is a compact set. This is the main reason why we need $\Gamma$ to be convex-cocompact.

In order to prove that any ``strip'' along $U$ has infinite measure, we argue by contradiction: if some $\rho$-neighbourhood has finite measure with respect to the conditional measure of $\BMS$ along $N$, then this must hold almost surely for any neighbourhood as large as we like (because of the self-similarity of the conditional measures). In particular we can project these conditional measures onto $N/U$ and end up with a family of Radon measures. These ``transversal'' Radon measures must still have dimension $\delta_\Gamma=n-m$ (this was shown in \cite{preprintOM}), and this implies in turn that they must be the Lebesgue measure of $N/U$. On the other hand, the Besicovitch-Federer projection Theorem implies that the projection of the conditional measures onto $N/U$ must be singular with respect to the Lebesgue measure, because the conditional measure are purely unrectifiable. Hence our Theorem is proved.

The push-forward of the Borel measure $\mu$ through the Borel function $f$ is denoted by $f\mu$; thus $f\mu (A)=\mu(f^{-1}(A))$ for any Borel set $A$.

For any set $E$, we denote by $\mathbf{1}_E$ the characteristic function:
\[ \mathbf{1}_E(x) = \left\{ \begin{array}{lll} 1& \mathrm{if} & x \in E \\ 0 & \mathrm{if} & x \notin E
  \end{array}
   \right.  \]

\section{Proof of the main theorem}
\subsection{Preliminary setup}
In order to study the BR measure with respect to some $m$-plane $U$ in
$N$, it is useful to look at the geometry of the BMS measure with
respect to the foliation induced by $U$ in the $N$-orbits (more
precisely, with respect to the projection along this foliation).

The technical tool that allows this is disintegration of measures.

Since we are going to apply tools from classical geometric measure
theory, we want to work with measures living on $N$ (recall that $N$
identifies with the Euclidean space $\R^n$). To $\BMS$-almost every $x
\in \Gamma \backslash G$ we are going to associate a measure (more
precisely, a \emph{projective} measure, \emph{i.e.} a measure
\emph{modulo} a positive scalar) $\sigma(x)$ on $N$ that reflects the
geometry of $\BMS$ along the unstable horosphere passing through $x$.

We now set up the needed formalism. The operation of $N$ on $G$ (on
the right) is smooth (\emph{i.e.} the quotient Borel space $G/N$ is a
standard Borel space). Lift $\BMS$ (which lives on $\Gamma \backslash
G$) to $G$; the measure we get is a $\Gamma$-invariant Radon measure
$\tilde{m}_{\mathrm{BMS}}$. Disintegrate this measure along $N$; for
almost every $g \in G$ we thus get a measure $m_{gN}$ supported on
$gN$ (see \cite{NegCurv} section 3.9 for a description of this
measure).

In general when disintegrating an infinite measure, the conditional
measures are canonically defined only up to a (non-zero) scalar; in
fact here there is a way to normalize them in a canonical way (by
introducing an appropriate measure on the space of horospheres, more
precisely this space lifted by $M$) but this would not be useful for
our purpose. See \emph{e.g.} \cite{Roblin}.

We now want to look at measures on $N$ instead of measures on $G$. For
any $g \in G$, there is a mapping $\phi_g : N \to G$ which
parametrizes the ``unstable horosphere'' $H^+(g)=gN$ in the usual way:
$\phi_g(h)=gh$ for any $h \in N$.

Since $\tilde{m}_{\mathrm{BMS}}$ is $\Gamma$-invariant, the pull-back
measures
\[ (\phi_g)^{-1} (m_{gN}), \quad (\phi_{\gamma g})^{-1} (m_{\gamma g
  N}) \] (which live on $N$) are equal up to a scalar multiple, for
$\tilde{m}_{\mathrm{BMS}}$-almost every $g \in G$ and every $\gamma
\in \Gamma$.

Let $\mathcal{M}_{\mathrm{rad}}(N)$ be the space of positive Radon
measures on $N$ and $\mathcal{M}_{\mathrm{rad}}^1(N)$ be the space of
projective classes of Radon measures on $N$, that is, the quotient of
$\mathcal{M}_{\mathrm{rad}}(N)$ by the equivalence relation
\[ \mu \sim \nu \Leftrightarrow \nu = t \mu,\quad t>0 \text. \]

We define a mapping $\sigma : \Gamma \backslash G \to
\mathcal{M}_{\mathrm{rad}}^1(N) $ by letting $\sigma(x)$ be the
projective class of
\[ (\phi_g)^{-1} (m_{gN}) \] if $x = \Gamma g$. This is well-defined
$\BMS$-almost everywhere.

We say that $\sigma$ is obtained by \emph{disintegrating $\BMS$ along
  $N$. }

This is a particular instance of the general theory of conditional
measures along a group operation, see \cite{preprint} or \cite{thesis}
(Chapter 2).

We record the following facts which we will use freely throughout this
paper:
\begin{lemma}\label{lemma.collect}
  \begin{enumerate}
  \item If some Borel subset $\Omega \subset \Gamma \backslash G$ has
    full $\BMS$-measure, then for $\BMS$-almost every $x$, the set
    \[\{ h \in N\ ;\ xh \in \Omega \}\] has full $\sigma(x)$-measure.

  \item There is a Borel subset $X \subset \Gamma \backslash G$ of
    full $\BMS$-measure such that if $x \in X$ and $h_0 \in H$ are
    such that $xh_0 \in X$, then $\sigma(xh_0)$ is the push-forward of
    $\sigma(x)$ through left translation by $h_0$ in $N$,
    \[ h \mapsto h_0h \text. \]
  \item For $\BMS$-almost every $x \in \Gamma \backslash G$, the
    origin of $N$ belongs to the support of $\sigma(x)$.
  \item For any $t \in \R$ and $\BMS$-almost every $x \in \Gamma
    \backslash G$,
    \[ \sigma(x a_t) = S_t \sigma(x) \] \emph{i.e.} $\sigma(x a_t)$ is
    the push-forward of $\sigma(x)$ through the ghomothety $S_t : N
    \to N$.
  \item For any $m \in M$, and $\BMS$-almost every $x \in \Gamma
    \backslash G$, $\sigma(xm)$ is the push-forward of $\sigma(x)$
    through the mapping $h \mapsto mhm^{-1}$. (Recall that the
    operation of $M$ by conjugation on $N$ identifies with the
    canonical operation of $\mathbf{SO}(n)$ on $\R^n$.)
  \item For $\BMS$-almost every $x \in \Gamma \backslash G$ and
    $\sigma(x)$-almost every $h\in N$,
    \[ 0 < \liminf_{\rho \to 0}
    \frac{\sigma(x)(B(h,\rho))}{\rho^{\delta_\Gamma}} \leq
    \limsup_{\rho \to 0}
    \frac{\sigma(x)(B(h,\rho))}{\rho^{\delta_\Gamma}} < \infty
    \text. \]
    
  \end{enumerate}
\end{lemma}
\begin{proof}
  Statements 1, 2 and 3 are clear. Statement 4 holds because of
  invariance of $\BMS$ with respect to the geodesic flow and formula
  \eqref{form.equivariance}. Statement 5 holds because $\BMS$ is
  $M$-invariant by definition. Statement 6 holds because $\Gamma$ is
  convex-cocompact and $\sigma(x)$ is equivalent to the
  Patterson-Sullivan measure; see \cite{Coornaert}, Proposition 7.4
  and \cite{NegCurv}, section 3.9

\end{proof}

\begin{notation}If $\mu$ is a Borel measure or projective measure on
  $N$, the support of which contains the origin on $N$, we let
  \[ \mu^* = \frac{\mu}{\mu(B_1)} \] \emph{i.e.} $\mu^*$ is the
  measure colinear to $\mu$ that gives measure $1$ to the unit ball
  $B_1$.

  We also denote by $S_t^* \mu$ the measure $(S_t \mu)^*$.
\end{notation}
In particular, since for $\BMS$-almost every $x \in \Gamma \backslash
G$, the origin of $N$ belongs to the support of $\sigma(x)$, we denote
by $\sigma^*(x)$ the Radon measure on $N$ that belongs to the
projective class $\sigma(x)$ and such that the unit ball $B_1\subset
N$ has measure $1$:
\[ \sigma^*(x)(B_1)=1 \text. \]

We denote by $\mathrm{Dirac}(x)$ the Dirac mass at $x$, \emph{i.e.}
the probability measure giving measure $1$ to $\{x\}$. Associated to
$\BMS$ is the following probability measure on the space of Radon
measures on $N$:
\begin{equation} P = \int_{\Gamma \backslash G} \mathrm{d}\BMS(x)\
  \mathrm{Dirac}(\sigma^*(x))\text. \label{eq.P}
\end{equation}
Recall that we assume that $\Gamma$ is Zariski-dense and has finite
BMS measure, so that $P$ is an Ergodic Fractal Distribution (EFD) in
the sense of Hochman (see \cite{Hoch}, Definition 1.2, and
\cite{preprintOM}, Lemma 5.3 for a proof that P is indeed an EFD).

\subsection{Unrectifiability of the limit set}
Recall that a Radon measure $\mu$ on the Euclidean space $\R^n$ is said to be purely $m$-unrectifiable if for any Lipschitz mapping $f : \R^m \to \R^n$, the range $f(\R^m)$ has measure zero with respect to $\mu$.

Assume that the growth exponent $\delta_\Gamma$ is an integer $< n$. The fact that the limit set of $\Gamma$ is purely $\delta_\Gamma$-unrectifiable when $\Gamma$ is convex-cocompact and Zariski-dense (the latter hypothesis is obviously necessary) is probably well-known, and certainly very intuitive. We give a full proof of this fact as it is pivotal in our argument.

\begin{proposition} \label{prop.unrectifiable}
Assume that $\Gamma$ is convex-cocompact and Zariski-dense. If $\delta_\Gamma$ is an integer strictly smaller than $n$, the conditional measure $\sigma(x)$ is almost surely purely $\delta_\Gamma$-unrectifiable.
\end{proposition}

\begin{proof}
Let $\Omega$ be the set of all  $x \in \Gamma \backslash G$ such that
  \[ \frac{1}{T} \int_0^T \mathrm{Dirac}(S_t^* \sigma(x)) \mathrm{d} t \]
  converges weakly to $P$ (recall equation \eqref{eq.P}) as $T \to +\infty$. This set has full BMS measure (\cite{preprintOM}, Lemma 5.4). Now fix some $x_0 \in \Omega$ such that for $\sigma(x_0)$-almost every $h \in N$, $x_0h \in \Omega$ (see Lemma \ref{lemma.collect}).

  We argue by contradiction. Assume that  some subset $L \in N$ is the image of a Lipschitz mapping $\R^{\delta_\Gamma} \to N$ and satisfies
  \[ \sigma(x_0)(L)>0\text.\]
  Note that the restriction $\sigma(x_0)|L$, which we denote by $\sigma_L(x_0)$, is $\delta_\Gamma$-rectifiable, and satisfies
  \[ 0 < \liminf_{\rho \to 0} \frac{\sigma_L(x_0)(B(h,\rho))}{\rho^{\delta_\Gamma}} \leq \limsup_{\rho \to 0} \frac{\sigma_L(x_0)(B(h,\rho))}{\rho^{\delta_\Gamma}} < \infty \]
  for $\sigma_L(x_0)$-almost every $h$ (Lemma \ref{lemma.collect}). By virtue of \cite{Mattila}, Theorem 16.7 and Lemma 14.5, for $\sigma_L(x_0)$-almost every $h$, there is a $\delta_\Gamma$-plane $V(h)$ such that
  \[ S_{t}^* \sigma(x_0h) \]
  converges weakly to the Haar measure on $V(h)$ as $t \to \infty$ 

  Recall that for $\sigma(x_0)$-almost every $h$,
  \[ \frac{1}{T} \int_0^T \mathrm{Dirac}(S_t^* \sigma(x_0 h)) \; \mathrm{d} t \]
  also converges weakly to $P$ as $T$ goes to infinity.

  We thus see that $P$-almost every $\mu$ is the Haar measure on some $\delta_\Gamma$-plane.   In other words, for $\BMS$-almost every $x$ the conditional measure at $x$,  $\sigma(x)$,  is concentrated on some $\delta_\Gamma$-plane of $N$; this contradicts the fact that the support of $\sigma(x)$ must be Zariski-dense, since $\Gamma$ is Zariski-dense. Hence the proposition.
\end{proof}
\begin{corollary}
Under the same hypotheses, the limit set $\Lambda_\Gamma$ is purely $\delta_\Gamma$-unrectifiable.
\end{corollary}
Recall that the limit set $\Lambda_\Gamma$ is the set of accumulation points of $\Gamma$ in $\Hr^{n+1} \cup \partial \Hr^{n+1}$. It is locally bilipschitz equivalent to the support of $\sigma(x)$ for $\BMS$-almost every $x$, so that the corollary follows readily from the proposition.

\subsection{The conditional measures are transversally singular}

\begin{proposition} \label{prop.singular}
Assume that $\Gamma$ is Zariski-dense and convex-cocompact and that $\delta_\Gamma=n-m$ where $m$ is an integer, $1 \leq m \leq n-1$. Fix some $m$-plane $U$ in $N$.

For $\BMS$-almost every $x \in \Gamma \backslash G$, the push-forward of the conditional measure $\sigma(x)$ through the canonical projection $N \to N/U$ is singular with respect to the Lebesgue measure on $N/U$.
\end{proposition}
Recall that a measure $\mu$ is singular with respect to a measure $\nu$ if it gives full measure to a $\nu$-negligible set.

\begin{proof}

  For any $m$-plane $V$, denote by $\pi_V$ the canonical projection $N \to N/V$.

  We will show that there exists an $m$-plane $U_0$ such that for almost every $x$, the push-forward of $\sigma(x)$ through $N \to N/U_0$ is singular with respect to the Lebesgue measure on $N/U$. Since the BMS measure is $M$-invariant, this implies that the same statement holds for any other $m$-plane $U$.

  According to Lemma \ref{lemma.lset} and the previous Propostion, for $\BMS$-almost every $x$ there is a sequence of Borel sets $(A_k)_k$ such that
  \begin{itemize}
  \item $\cup_k A_k$ has full $\sigma(x)$-measure,
  \item  each $A_k$ has finite $(n-m)$-dimensional Hausdorff measure,
  \item and each $A_k$ is purely $(n-m)$-unrectifiable.
  \end{itemize}

  By virtue of the Besicovitch-Federer projection theorem (\cite{Mattila}, Theorem 18.1 (2)), the image of $\cup_k A_k$ in $N/V$ is Lebesgue-negligible for almost every $m$-plane $V$ (with respect to the Haar measure on the Grassmannian of $m$-planes in $N$). This shows that for almost every $m$-plane $V$, the push-forward of $\sigma(x)$ through $\pi_V$ is singular with respect to the Lebesgue measure.

  This holds for almost every $x$. A standard application of Fubini's theorem now yields that there exists an $m$-plane $U_0$ such that for almost every $x$, the push-forward of $\sigma(x)$ through $\pi_{U_0}$ is singular with respect to the Lebesgue measure. The proposition is thus proved.
 \end{proof}
\begin{lemma} Assume that $\Gamma$ is convex-cocompact. For $\BMS$-almost every $x \in \Gamma \backslash G$, $\sigma(x)$ is supported by a countable union of $\delta_\Gamma$-sets. \label{lemma.lset}
\end{lemma}
Recall that $E$ is a $\delta$-set if its $\delta$-dimensional Hausdorff measure is finite and non-zero.
\begin{proof}
It is well-known (see \cite{Sullivan1}, Theorem 7) that the limit set $\Lambda_\Gamma$ is a $\delta_\Gamma$-set. Since it is (almost surely) locally bilipschitz-equivalent to the support of $\sigma(x)$, the lemma follows.
\end{proof}

\subsection{Conditional measure  of  strips}
If $U$ is any $m$-plane in $N$ ($1 \leq m \leq n-1$), we denote by $B_\rho^T(U)$ the $\rho$-neighbourhood of $U$ in $N$, that is the set of all $h \in N$ such that
\[ d(h,U) < \rho \text. \]
When it is clear from the context which $m$-plane we are talking about, we dispense ourselves with the letter $U$ in the notation.

\begin{proposition} \label{prop.stripinfinite}
  Assume that $\Gamma$ is convex-cocompact and Zariski-dense and that  $\delta_\Gamma = n-m$ where $m$ is an integer, $1 \leq m \leq n-1$. Fix some $m$-plane $U$ in $N$. For $\BMS$-almost every $x \in \Gamma \backslash G$ and any $\rho > 0$,
  \[ \sigma(x)(B_\rho^T) = \infty \text.\]
\end{proposition}
\begin{proof}
  It is enough to show that for any $\rho >0$, and almost every $x \in \Gamma \backslash G$, $\sigma(x)(B_\rho^T) = \infty$ (see lemma \ref{lemma.collect}). We argue by contradiction and assume that the set of those $x$ such that
  \[ \sigma(x)(B_\rho^T) < \infty \]
  has positive BMS measure; it must then have full measure since $\BMS$ is mixing and because of Lemma \ref{lemma.collect}.4. 

  It is easy to see then that for $\BMS$-almost every $x \in \Gamma \backslash G$,  \[ \sigma(x)(B_\rho^T) < \infty \]
  for any $\rho > 0$.

  This implies that  the push-forward of $\sigma(x)$ through the projection $\pi_U: N \to N/U$ is a projective \emph{Radon} measure.

  Now consider the distribution
  \[ P^T = \int \mathrm{d}m(x)\ \mathrm{Dirac}((\pi_U \sigma(x))^*) \]
  on the space of Radon measures on $N/U$.
  It is straight-forward to check that $P^T$ is an Ergodic Fractal Distribution (see \cite{preprintOM}, Lemma 5.3). Since $P^T$ has dimension $n-m$ (see \cite{preprintOM}, Theorem 4.1) this is possible only if
  \[ P^T = \mathrm{Dirac}(\mathrm{Haar}_{N/U}) \]
  \emph{i.e.} $P^T$ is the Dirac mass at the Haar measure of $N/U$.

  We are using the fact that a Fractal Distribution of dimension $d$ on some Euclidean space $\R^d$ has to be the only one we can think of, \emph{i.e.} $\mathrm{Dirac}(\mathrm{Haar}_{\R^d})$. In essence, this fact goes back to Ledrappier-Young (\cite{LY}, Corollary G). In the setting of Fractal Distributions it was proved by Hochman in \cite{Hoch}, Proposition 6.4 (see also \cite{MargulisTomanov}).

Now we end up with the conclusion that for $\BMS$-almost every $x \in \Gamma \backslash G$, the push-forward of $\sigma(x)$ through $\pi_U$ is the Haar measure on $N/U$; this contradicts Proposition \ref{prop.singular}. Hence the proposition is proved.
\end{proof}
\begin{remark}
  Propositions \ref{prop.unrectifiable}, \ref{prop.singular} and \ref{prop.stripinfinite} admit obvious counter-examples when $\Gamma$ is not Zariski-dense: take some lattice $\Gamma \subset \mathbf{SO}^o(1,m+1)$ and look at the image of $\Gamma$ through the embedding
  \[ \mathbf{SO}^o(1,m+1) \to \mathbf{SO}^o(1,n+1) \text. \]
\end{remark}

\subsection{Recurrence of the Burger-Roblin measure}
We are now ready to prove our main theorem. We use the following consequence of proposition \ref{prop.stripinfinite}.

\begin{lemma} \label{lemma.base}
Assume that $\Gamma$ is Zariski-dense and convex-cocompact and thatand $\delta_\Gamma=n-m$. Fix an $m$-plane $U$ in $N$. Let $\Omega_\Gamma$ be the support of the Bowen-Margulis-Sullivan measure in $\Gamma \backslash G$. For almost every $x \in \Gamma \backslash G$, and any $\rho > 0$, the set of all $h \in B^T_\rho(U)$ such that $xh \in \Omega_\Gamma$ is unbounded.
\end{lemma}
\begin{proof}
  By construction of the disintegration mapping $\sigma$, the support of $\sigma(x)$, $\mathrm{supp}(\sigma(x))$, is almost surely the set of all $h \in N$ such that $xh$ belongs to $\Omega_\Gamma$.  Since the \emph{Radon} measure $\sigma(x)$ gives infinite measure to $B^T_\rho(U)$, the intersection $B^T_\rho \cap \mathrm{supp}(\sigma(x))$ must be unbounded; hence the lemma.
\end{proof}

\begin{proposition}
  Assume that $\Gamma$ is Zariski-dense and convex-cocompact and that $\delta_\Gamma=n-m$. Fix some $m$-plane $U$ in $N$. For BMS-almost every $x$, there is a compact $K \subset \Gamma \backslash G$ such that
  \[ \int_U \mathbf{1}_K(xu) \mathrm{d}u = \infty \text. \]
  Furthermore, if $W$ is any neighbourhood of $\Omega_\Gamma$, $K$ may be chosen inside $W$.
\end{proposition}
Of course $U$ is endowed with the Haar measur in this formula.
\begin{proof}
  First of all, recall that $\Omega_\Gamma$ is a compact subset of $\Gamma \backslash G$ since $\Gamma$ is convex-cocompact.

  For any $\rho>0$, let $K_\rho$ be the set of all $xh$ where $x \in \Omega_\Gamma$ and $h$ belongs to the closed $\rho$-ball centered at the origin in $N$. This is again a compact set. If $\rho$ is small enough, $K_\rho$ is  a subset of $W$. Fix such a $\rho$.

  By lemma \ref{lemma.base}, we may find a sequence $(h_k)_k$ of elements of $B_\rho^T(U)$ that goes to infinity and such that $xh_k \in \Omega_\Gamma$ for any $k$; if we let $h_k=u_k v_k$ where $u_k \in U$ and $v_k$ is orthogonal to $U$, we have
  \[ x u_k \in K_\rho \]
  for any $k$, and the sequence $(u_k)_k$ goes to infinity.

  According to lemma \ref{lemma.general}, we may thicken $K_\rho$ to get a compact set $K \subset W$, such that the conclusion of the proposition holds.
\end{proof}
\begin{remark}
  It is necessary to consider a compact set $K$ that is slightly bigger than $\Omega_\Gamma$ in this lemma, since by virtue of Proposition \ref{prop.unrectifiable}, one has
  \[ \int_U \mathbf{1}_{\Omega_\Gamma} (xu) \mathrm{d}u = 0 \]
  for BMS-almost every $x$.
\end{remark}

\begin{corollary}\label{cor.principal}
  Under the same hypothesis, for BR-almost every $x$ there is a compact $K$ such that
  \[ \int_U \mathbf{1}_K (xu) \mathrm{d}u = \infty \text. \]
  In particular, the BR measure is recurrent with respect to $U$.
\end{corollary}
\begin{proof}
  The set of all $x \in \Gamma \backslash G$ that satisfy the conclusion is obviously $N$-invariant; since it has full BMS measure, it must have full BR measure as well.
\end{proof}

The following lemma is well-known but I have not been able to pinpoint a proof in the literature. We need it only when $G$ is some $\R^m$ but there is no reason not to prove it in full generality.
\begin{lemma} \label{lemma.general}
  Let $X$ be some second countable locally compact space where a second countable locally compact topological group $G$ acts continuously. Assume that we are given some  fixed $x_0 \in X$ and a sequence $(g_n)_n$ in $G$ that goes to infinity, such that $g_n x_0$ belongs to a fixed compact subset $K$ for every $n$. Then for any neighbourhood $W$ of $K$, there is a compact subset $L$ of $W$ such that
  \[ \int \mathbf{1}_K (gx_0)\ \mathrm{d}g = \infty \text. \]
\end{lemma}
Here  $G$ is endowed with some \emph{right-invariant} Haar measure.
\begin{proof}
Endow $X$ with some compatible metric; endow $G$ with some compatible metric that is also right invariant and proper (which means that closed balls are compact), see \cite{Struble}. 

Fix some $\delta>0$ small enough that the set
\[ L = \{ x \in X\ ;\ d(x,K) \leq \delta \} \]
is compact and contained in $W$.

For any $n \geq 1$, let $\varepsilon_n$ be the lower bound of the set of all $\varepsilon >0$ for which the closed ball $B(g_n,\varepsilon)$ contains some $h$ such that $d(g_n x_0, hx_0)=\delta$. If there is no such $\varepsilon$, then the whole orbit $G x_0$ is contained in $L$ and the proof is over. We may thus assume that $\varepsilon_n < \infty$ for every $n$. It is clear also that $\varepsilon_n > 0$.

We now prove that $\inf_n \varepsilon_n > 0$. The mapping from $B_1 \times K$ to $\R$ (where $B_1$ is the closed unit ball in $G$)
\[ (g,y) \mapsto d(y,gy) \]
is uniformly continuous because it is continuous and $B_1 \times K$ is compact. In particular there is some $\eta \in ]0,1[$ such that the relation $d(g,e) < \eta$ (where $e$ is the unit of $G$) implies
\[ d(gy,y) < \delta \]
for any $y \in K$.

We are going to show that $\varepsilon_n \geq \eta$ for any $n$. Let $h$ be any element of $G$ such that $d(g_n,h) < \eta$; then $d(h g_n^{-1},e) < \eta$ (because the distance on $G$ is right-invariant) which implies
\[ d(h x_0, g_n x_0) = d(hg_n^{-1} g_n x_0,g_n x_0) < \delta \text. \]
By definition of $\varepsilon_n$, this is means that $\varepsilon_n \geq \eta$. Hence $\inf_n \varepsilon_n > 0$.

Now pick some positive $\varepsilon$ smaller than $\inf_n \varepsilon_n$. If $h \in G$ is such that $d(h,g_n) \leq \frac{\varepsilon}{2}$, then $d(hx_0,gx_0) < \delta$, so that $hx_0 \in L$.

This shows that the orbital mapping $\rho_{x_0}: g \mapsto gx_0$ maps each $B(g_n,\varepsilon/2)$ inside $L$. As $g$ goes to infinity in $G$, we may, passing to a subsequence, assume that these balls are pair-wise disjoint. Their union has infinite Haar measure because the metric on $G$ is right-invariant. Whence the lemma.
\end{proof}

\subsection{Return rate}

In the following proposition we let $B_R$ be the $R$-ball centered at the origin in $N$ and as previously $B_\rho^T$ is the $\rho$-neighbourhood of the $m$-plane $U$ in $N$.

We do not assume that $\Gamma$ is convex-cocompact nor that $\delta_\Gamma$ is an integer.
\begin{proposition}
  Assume that $\Gamma$ has finite BMS measure and is Zariski-dense. Let $m$ be some integer, $1 \leq m \leq n-1$. Fix an $m$-plane $U$ in $N$. For all $\rho>0$ and almost every $x \in \Gamma \backslash G$,
  \[ \liminf_{R \to \infty} \frac{\log \sigma(x)(B_\rho^T \cap B_R)}{\log R} \leq \sup\{ 0, \delta_\Gamma - (n-m) \} \text. \]
\end{proposition}
\begin{remark}
It is not clear whether one should expect the lower limit in this proposition to be a genuine limit.
\end{remark}

\begin{proof}
  Recall the following:
  \begin{itemize}
    \item for almost every $x$ and every fixed $R>0$,
  \[ \lim_{\rho \to 0} \frac{\log \sigma(x)(B_\rho^T \cap B_R)}{\log \rho} = \inf\{ n-m, \delta_\Gamma \} \]
\item for almost every $x$,
  \[ \lim_{R \to +\infty} \frac{\log \sigma(x)(B_R)}{\log R} = \delta_\Gamma \text. \]
\end{itemize}
The first limit comes from the fact that the projection of $\sigma(x)|B_R$ onto $N/U$ has exact dimension $\inf \{n-m,\delta_\Gamma\}$ (see \cite{preprintOM}, Theorem 4.1). The second limit holds because $\BMS$ is ergodic with respect to the automorphism $a_t$ for $t>0$ as well as for $t<0$; thus,
\[ \lim_{R \to \infty} \frac{\log \sigma(x)(B_R)}{\log R} = \lim_{r \to 0} \frac{\log \sigma(x)(B_r)}{\log r} = \delta_\Gamma \]
see \cite{thesis}, Lemme 2.2.1.

Let us denote by $\theta$ the number $\inf \{ n-m,\delta_\Gamma \}$. Fix some $\varepsilon > 0$.

  For $\BMS$-almost every $x$, there is some $\rho_0(x) > 0$ such that the relation $\rho \leq \rho_0(x)$ implies that
  \[  \sigma^*(x)(B_\rho^T \cap B_1) \leq \rho^{\theta - \varepsilon} \text. \]
  Choose $\rho_0>0$ small enough that the set $E_{\rho_0}$ of all $x$ such that $\rho_0(x)> \rho_0$ has positive BMS measure. Let $a=a_t$ for some $t>0$.

  For $\BMS$-almost every $x$, one can find arbitrarily big integers $k$ such that $a^k x \in E_{\rho_0}$ (because $\BMS$ is $a$-ergodic). If $k$ is such an integer, we have
  \[ \frac{ \sigma^*(x) (B^T_{e^k \rho} \cap B_{e^k})}{ \sigma^*(x)(B_{e^k})} \leq \rho^{\theta - \varepsilon} \]
  for any $\rho \leq \rho_0$ (Lemma \ref{lemma.collect}.4).

  Assume, furthermore, that $k$ is so large that $ \sigma^*(x)(B_{e^k}) \leq e^{k (\delta_\Gamma + \varepsilon)}$, and that $e^{-k} < \rho_0$. Letting $\rho=e^{-k}$, we get
  \[  \sigma^*(x) (B^T_1 \cap B_{e^k}) \leq e^{-k(\theta - \varepsilon)} e^{k(\delta_\Gamma + \varepsilon)} = e^{k (\delta_\Gamma - \theta + 2 \varepsilon)} \text. \]

  Since $k$ can be as large as we like, this shows that
  \[ \liminf_{k \to \infty} \frac{\log \sigma(x)(B_1^T \cap B_{e^k})}{k} \leq \sup\{0,\delta_\Gamma - (n-m)\} + 2 \varepsilon \]
  for any $\varepsilon > 0$. The lemma follows.

\end{proof}

\begin{corollary}
  Assume that $\Gamma$ is convex-cocompact and Zariski-dense. Let $m$ be an integer, $1 \leq m \leq n-1$. For any $m$-plane $U$ in $N$, and any compact $K$ in $\Gamma \backslash G$,
  \[ \liminf_{R \to \infty} \frac{\log \left( \mathrm{Haar}_U ( \{ u \in B_R\ ;\ xu \in K \}) \right)}{\log R} \leq \sup\{ 0, \delta_\Gamma - (n-m) \} \]
  for $\BMS$-almost every $x$ and also for $m_{\mathrm{BR}}$-almost every $x$.
\end{corollary}
We skip the straight-forward proof.

\bibliography{bibli}

\end{document}